\documentclass[journal]{IEEEtran}

\usepackage{epsfig,color,amsmath,cite}
\usepackage{amsthm} 
\usepackage{amsmath}    
\IEEEoverridecommandlockouts
\usepackage{bm}
\usepackage{epstopdf}
\usepackage{amssymb}
\usepackage{url}
\usepackage{enumitem} 
\usepackage{multirow}
\usepackage{hhline}
\usepackage{booktabs}
\usepackage[linesnumbered,boxed,commentsnumbered,ruled,vlined,longend]{algorithm2e}
\usepackage{comment}

\DeclareMathOperator*{\subjectto}{subject\ to}

\makeatother
\DeclareMathAlphabet\mathbfcal{OMS}{cmsy}{b}{n}

\newtheorem{theorem}{Theorem}
\newtheorem{mydef}{Definition}

\newtheorem{myrem}{Remark}

\newtheorem{myprs}{Proposition}

\newcommand{\reducewordspace}[0]{\fontdimen2\font=0.645ex}



\usepackage{stackengine}
\newcommand\barbelow[1]{\stackunder[1.2pt]{$#1$}{\rule{.8ex}{.075ex}}}

\newcommand{\mat}[1]{\boldsymbol{#1}}

\newcommand{\bmat}[1]{\begin{bmatrix} #1 \end{bmatrix}}

\providecommand{\mA}{\ensuremath{\mat{A}}}
\providecommand{\mB}{\ensuremath{\mat{B}}}

\providecommand{\mG}{\ensuremath{\mat{G}}}

\newcommand{\m}{\boldsymbol}
\allowdisplaybreaks[4]
\usepackage[colorlinks = true,
linkcolor = blue,
urlcolor  = blue,
citecolor = blue,
anchorcolor = blue]{hyperref}

\newcommand{\mbb}[1]{\mathbb{#1}}

\usepackage[framemethod=TikZ]{mdframed}
\mdfdefinestyle{MyFrame}{%
	linecolor=black,
	outerlinewidth=1.25pt,
	roundcorner=1.25pt,
	innerrightmargin=5pt,
	innerleftmargin=5pt,}
	
\usepackage[noabbrev]{cleveref}

\usepackage{mathtools}

\DeclarePairedDelimiter\abs{\lvert}{\rvert}%
\DeclarePairedDelimiter\norm{\lVert}{\rVert}%

\makeatletter
\let\oldabs\abs
\def\abs{\@ifstar{\oldabs}{\oldabs*}}
\let\oldnorm\norm
\def\norm{\@ifstar{\oldnorm}{\oldnorm*}}
\makeatother

\usepackage[english]{babel}
\usepackage[utf8]{inputenc}
\usepackage[super]{nth}

\usepackage{graphicx}
\usepackage{float}
\usepackage[caption = false]{subfig}

\usepackage{array}
\usepackage{threeparttable}

\usepackage[english]{babel}
\usepackage[utf8]{inputenc}
\usepackage[super]{nth}

\usepackage{pifont}
\title{\Large \vspace{0.4cm} \LARGE \centering {\textbf{New Insights on One-Sided Lipschitz and Quadratically-Inner Bounded Nonlinear Dynamic Systems}}}

\author{Sebastian A. Nugroh$\text{o}^{\dagger}$, Vu Hoan$\text{g}^{\ddagger}$, Maria Rados$\text{z}^{\ddagger}$, Shen Wan$\text{g}^{\dagger}$, and Ahmad F. Tah$\text{a}^{\dagger}$
	\thanks{$^{\dagger}$Department of Electrical and Computer Engineering, The University of Texas at San Antonio, One UTSA Circle, San Antonio, TX 78249. Emails: sebastian.nugroho@my.utsa.edu, mvy292@my.utsa.edu, ahmad.taha@utsa.edu}  
	\thanks{$^{\ddagger}$Department of Mathematics, The University of Texas at San Antonio, One UTSA Circle, San Antonio, TX 78249. Emails: duynguyenvu.hoang@utsa.edu, maria\_radosz@hotmail.com.}
	\vspace*{-0.1cm}
		\thanks{This work is partially supported by Valero Energy Corporation and National Science Foundation (NSF) under Grant CMMI-1728629, CMMI-1917164, DMS-1614797, and DMS-1810687.}
	\vspace{-0.4cm}
}

\begin{document}
	
\setlength{\abovedisplayskip}{3.5pt}
\setlength{\belowdisplayskip}{3.5pt}
\setlength{\abovedisplayshortskip}{3.1pt}
\setlength{\belowdisplayshortskip}{3.1pt}
		
\newdimen\origiwspc%
\newdimen\origiwstr%
\origiwspc=\fontdimen2\font
\origiwstr=\fontdimen3\font

\fontdimen2\font=0.65ex

\maketitle

\begin{abstract}
Nonlinear dynamic systems can be classified into various classes depending on the modeled nonlinearity. These classes include Lipschitz, bounded Jacobian, one-sided Lipschitz (OSL), and quadratically inner-bounded (QIB). Such classes essentially yield bounding constants characterizing the nonlinearity. This is then used to design observers and controllers through Riccati equations or matrix inequalities. While analytical expressions for bounding constants of Lipschitz and bounded Jacobian nonlinearity are studied in the literature, OSL and QIB classes are not thoroughly analyzed---computationally or analytically. In short, this paper develops analytical expressions of OSL and QIB bounding constants. These expressions are posed as constrained maximization problems, which can be solved via various optimization algorithms. 
This paper also presents a novel insight particularly on QIB function set: any function that is QIB turns out to be also Lipschitz continuous.
\end{abstract}



\section{Introduction and Paper's Contribution}
In the past few decades, hundreds of control-theoretic studies have investigated designing observer/controller for nonlinear dynamic systems (NDS) which can generally be expressed as
\begin{align}\label{eq:gen_dynamic_systems_0}
\dot{\m x}(t) &= \m f(\m x, \m u),\;\; \;\; \m y(t) =\m h(\m x, \m u)
\end{align}
\noindent where $\m x\in \mathbfcal{X}\subset \mathbb{R}^n$ is the state, $\m u\in \mathbfcal{U}\subset \mathbb{R}^m$ is the input, $\m y\in\mathbb{R}^p$ is the output, and the mappings $\m f :\mathbb{R}^n\times \mathbb{R}^m\rightarrow \mathbb{R}^n$ and $\m h :\mathbb{R}^n\times \mathbb{R}^m\rightarrow \mathbb{R}^p$ represent nonlinearities in the NDS.
The majority of these observer/controller designs utilize either linear matrix inequalities (LMIs) formulations or algebraic Riccati equations~\cite{boyd1994linear,vanantwerp2000tutorial}. These designs almost always assume that the nonlinear function $\m f(\cdot)$ belongs to certain nonlinearity \textit{classes} or \textit{function sets}.
For example, observer designs for Lipschitz continuous nonlinearities have been developed in \cite{Phanomchoeng2010,ichalal2012observer,zemouche2006observer,Nugroho2019ITS,Nugroho2019TPS} while observer-based control and stabilization are proposed in studies \cite{Yadegar2018,zemouche2017robust,Ekramian2017}. 

Beyond the somewhat conservative Lipschitz assumption, the control-theoretic application of \textit{one-sided Lipschitz} (OSL) and \textit{quadratically inner-bounded} (QIB) function sets are introduced in \cite{Abbaszadeh2010}. 
In short, $\m f(\cdot)$ is Lipschitz, OSL, or QIB if the following conditions are satisfied

\vspace{-0.1cm}
\begin{align*}
\textbf{Lipschitz:} & \;\;\; \norm {\m f(\m x,\m u)-\m f(\hat{\m x},\m u)}_2 \leq  \gamma_l\norm {\m x - \hat{  \m x}}_2,\\
\textbf{OSL:} & \;\;\;  \left(\m f(\m x,\m u)-\m f(\hat{\m x},\m u)\right)^{\top}(\m x-\hat{\m x})\leq \gamma_s \norm{\m x - \hat{\m x}}_2^{2}, \\
\textbf{QIB:} & \;\;\; \norm {\m f(\m x,\m u)-\m f(\hat{\m x},\m u)}_2^{2}\leq  \gamma_{q1} \norm{\m x - \hat{\m x}}_2^{2} \\ &\;\;\;\;+\gamma_{q2}\left(\m f(\m x,\m u)-\m f(\hat{\m x},\m u)\right)^{\top}(\m x-\hat{\m x}),
\end{align*}
for a nonnegative Lipschitz constant $\gamma_l \in \mbb{R}^+$, OSL constant $\gamma_s \in \mbb{R}$,  QIB constants $\gamma_{q1}, \gamma_{q2} \in \mbb{R}$, where $(\m x, \m u),(\hat{  \m x}, \m u) \in \m \Omega$. To give examples in observer design, consider the following NDS and observer dynamics
\begin{align}
\dot{{\m x}}(t) &= \mA{  \m x} + \m f({\m x}, \m u) + \mB\m u(t) \\
\dot{\hat{\m x}}(t) &= \mA\hat{  \m x} + \m f(\hat{\m x}, \m u) + \mB\m u(t) + \m L\left(\m y(t) - \hat{  \m y}(t)\right), \label{eq:obs_dynamic_systems}
\end{align}
where $\m y(t) = \m C \m x(t)$, $\hat{  \m y}(t) = \m C \hat{  \m x}(t)$, and $ \m L$ is a Luenberger-like matrix variable. To guarantee stability of the estimation error dynamics, the observer design considering that $\m f(\cdot)$ is Lipschitz continuous translates to computing $\m P\succ 0$, $\m Y$, and $\epsilon>0$ such that \cite{Phanomchoeng2010}
	\begin{align}
\small \bmat{
	\m A ^{\top}\m P + \m P\m A - \m C ^{\top}\m Y ^{\top} -\m Y\m C  +\epsilon\gamma_l^2 \m I & * \\
	\m P & -\epsilon \m I} & \prec 0, \label{eq:LMI_lip_obs}
\end{align}
is feasible, where $\m L= \m P^{-1} \m Y$. If $\m f(\cdot)$ satisfies both OSL and QIB conditions, then the problem boils down to finding  $\m P\succ 0$, $\sigma,\epsilon_1,\epsilon_2> 0$ such that \cite{zhang2012full}
\begin{align}
\small \begin{bmatrix}
\m 	A^\top \bm P + \bm P\m A -\sigma \m C^\top \m C+ (\epsilon_1\gamma_s+\epsilon_2\gamma_{q1})\m I &* \\ 
\m P+\dfrac{\gamma_{q2}\epsilon_2-\epsilon_1}{2}\m I & -\epsilon_2 \m I  
\end{bmatrix} \prec 0, \label{eq:LMI_osl_qib_obs}
\end{align}
where the corresponding observer gain is $\m L = \frac{1}{2} \sigma \m P^{-1} \m C^{\top}$.
The Lipschitz constant $\gamma_l$ can only be nonnegative while OSL and QIB constants $\gamma_s$, $\gamma_{q1}$, and $\gamma_{q2}$ can at first sight be any real numbers. In particular, the authors in \cite{Abbaszadeh2010} demonstrate that in observer design, the OSL condition can be \textit{less conservative} compared to the Lipschitz condition. This stems from the fact that $\gamma_s$ can be negative. This in turn expands the feasibility region of the formulated matrix inequalities. 

Since then, many approaches have been developed in the literature to design observer/controller for NDS satisfying OSL and QIB. For instance, observer designs for these type of nonlinearities with various features are proposed in  \cite{benallouch2012observer,zhang2012full,zhang2015unknown}, while controller/observer-based stabilization methods are developed in \cite{wu2015observer,Rastegari2019,Gholami2019}. 

The aforementioned literature solely focuses on developing methods for observer/controller designs while assuming that the OSL and QIB constants are known. These studies demonstrate the applicability of the proposed approaches on small NDS where OSL and QIB constants can be obtained analytically. For large-scale NDS, computing these constants is nontrivial. To the best of our knowledge, there is an almost complete absence in the literature that is dedicated to \textit{parameterize} NDS---that is, computing the corresponding constants or \textit{parameters}---for OSL and QIB. This is in contrast to the more understood and studied Lipschitz nonlinearity; see \cite{SCHULZEDARUP2018135,Chakrabarty2019,Fazlyab2019,Nugroho2019}. These studies are presented for Lipschitz nonlinearity;  their extension to OSL and QIB function sets is unclear. 
To that end, the objective of this paper is to obtain analytical expressions of OSL and QIB constants. The contributions of our paper are as follows:
\begin{itemize}[leftmargin=*]
	\item We prove that any function satisfying the QIB condition is also Lipschitz continuous, which rectifies the current understanding of the QIB function set \cite{Abbaszadeh2010,Abbaszadeh2013,zhang2012full}. In addition, we derive an inequality constraint involving parameters $\gamma_{q_1}, \gamma_{q_2}$ that arises as a \textit{necessary} consequence of a function being QIB. This implies that $\gamma_{q_1}$  and $\gamma_{q_2}$ are constrained in order for the QIB condition to hold. Consequently, this contribution also results in correcting some unfortunate errors in the numerical examples produced in~\cite{Abbaszadeh2013,zhang2012full}. 
	\item We provide systematic methods to compute the corresponding constants/parameters for any NDS with arbitrary nonlinear models. The parameterization is procured through global, constrained maximization optimization problems. 
	\item We showcase the application of the proposed methods on a simple nonlinear system to compute its OSL and QIB parameters and utilize them for observer design. 
\end{itemize}
The paper organizations are summarized as follows.
The relation between Lipschitz continuity, OSL, and QIB, is discussed in Section \ref{sec:lip_osl_qib_relation}. 
In Section \ref{sec:nonlinear_classification}, we present the analytical parameterization of OSL and QIB function sets that both posed as global maximization problems. Section \ref{sec:interval_optimization} provides brief discussions on how to solve global, constrained maximization problems. Finally, Section \ref{sec:numerical_tests} provides a numerical example and Section \ref{sec:conclusion} concludes the paper.

\vspace{-0.0cm}
\noindent \textit{\textbf{Paper's Notation:}} 
For vectors $\m x, \m y\in\mbb{R}^n$, inner product is defined as $\langle\m x,\m y\rangle := \m x^{\top}\m y$. For matrix $\m A$, $A_{(i,j)}$ denotes its $i$-th and $j$-th element.  The set $\mathbb{I}(n)$ is defined as $\mathbb{I}(n) := \{i\in\mathbb{N}\,|\, 1\leq i \leq n\}$, which is usually used to represent the set of indices. The notation $\nabla_{\hspace{-0.05cm}x}$ defines the gradient vector with respect to vector $\m x$. For a square matrix $\mA$,  $\lambda_{\mathrm{max}}(\mA)$ and $\lambda_{\mathrm{min}}(\mA)$ return the maximum and minimum eigenvalues of $\mA$. 

%

\vspace{-0.1cm}
\section{Problem Formulation and Preliminaries\vspace{-0.05cm}}\label{sec:problem_formulation}
This paper focuses on parameterizing NDS \eqref{eq:gen_dynamic_systems} into OSL and QIB sets for the a general model of NDS:
\begin{align}\label{eq:gen_dynamic_systems}
\dot{\m x}(t) &= \mA \m x (t) +\mG\m f(\m x, \m u) + \mB\m u(t),\;\; \m y(t) = \m C \m x(t)
\end{align}
where $\m G\in\mbb{R}^{n_x\times n_g}$ describes the distribution of the nonlinearities.
 The set $\mathbf{\Omega}:= \mathbfcal{X}\times \mathbfcal{U}$ with $\mathbf{\Omega}\subset\mbb{R}^p$ is used throughout the paper to represent the domain of $\m f(\cdot)$, which is convex, compact. Moreover, we assume that $\mathbfcal{X}$ has a nonempty interior. 
This assumption is not restrictive as many dynamic systems in real world have bounded states and inputs---each upper and lower bounds define the operating regions. 
It is additionally assumed here that $\m f(\cdot)$ is differentiable and has continuous partial derivatives everywhere in $\mathbf{\Omega}$.
This ensures that its gradient vector $\nabla\m f(\cdot)$ is bounded within the set $\mathbf{\Omega}$ \cite{trench2003introduction}.  
As mentioned previously, in this paper we are seeking a way of computing constants such that $\m f(\cdot)$ satisfies OSL and QIB conditions---defined as follows.
\vspace{-0.1cm}
\begin{mydef}[OSL \& QIB]\label{def:osl_qib}
	\reducewordspace
The nonlinear function $\m f :\mathbb{R}^n\times \mathbb{R}^m\rightarrow \mathbb{R}^g$ in \eqref{eq:gen_dynamic_systems} is locally one-sided Lipschitz in $\mathbf{\Omega}$ if for any $(\m x, \m u), (\hat{\m x}, \m u)\in \mathbf{\Omega}$ it holds that
\begin{subequations}\label{eq:osl_qib_def}
	\vspace{-0.1cm}
	\begin{align}
	\langle \m G(\m f(\m x,\m u)-\m f(\hat{\m x},\m u)),\m x-\hat{\m x}\rangle\leq \gamma_s \norm{\m x - \hat{\m x}}_2^{2}, \label{eq:osl_def}
	\end{align}
	for $\gamma_s\in\mbb{R}$ and quadratically inner-bounded in $\mathbf{\Omega}$ if for any $(\m x, \m u), (\hat{\m x}, \m u)\in \mathbf{\Omega}$ it holds that
	\begin{align}
	&\langle\m G(\m f(\m x,\m u)-\m f(\hat{\m x},\m u)),\m G(\m f(\m x,\m u)-\m f(\hat{\m x},\m u))\rangle \leq \label{eq:qib_def}\\ &\qquad\gamma_{q1} \norm{\m x - \hat{\m x}}_2^{2} +\gamma_{q2}\langle\m G(\m f(\m x,\m u)-\m f(\hat{\m x},\m u)),\m x-\hat{\m x}\rangle, \nonumber
	\end{align}
for $\gamma_{q1},\gamma_{q2} \in \mbb{R}$. 
\end{subequations}	
\end{mydef}
\vspace{-0.1cm}
It is worthwhile mentioning that, since $(\m x,\m u)\in\mathbf{\Omega}$, the OSL and QIB conditions are both valid \textit{locally} (or \textit{semi-globally}) in the region of interest $\mathbf{\Omega}$.


\vspace{-0.0cm}
\section{The Relation Between Lipschitz, OSL, and QIB}\label{sec:lip_osl_qib_relation}
This section discusses the relations between Lipschitz continuity, OSL, and QIB. Without loss of generality, it is assumed throughout this section that $\m G = \m I$;  $\m I$ is the identity matrix.


\subsection{Lipschitz Continuity and OSL: A Known Result}\label{ssec:lip_osl_relation}
Lipschitz continuity is an important class of NDS as it is central for guaranteeing the existence and uniqueness of ODE solutions. This class is defined as follows.
\vspace{-0.1cm}
\begin{mydef}\label{def:lip}
	\reducewordspace
	The nonlinear function $\m f :\mathbb{R}^n\times \mathbb{R}^m\rightarrow \mathbb{R}^g$ in \eqref{eq:gen_dynamic_systems} is locally Lipschitz continuous in $\mathbf{\Omega}$ if for any $(\m x, \m u), (\hat{\m x}, \m u)\in \mathbf{\Omega}$ it holds that
\begin{align}
\norm{\m f(\m x,\m u)-\m f(\hat{\m x},\m u)}_2 \leq \gamma_l \norm{\m x - \hat{\m x}}_2,\label{eq:lip_def}
\end{align}
where $\gamma_{l} \geq 0$ is called the corresponding Lipschitz constant.
\end{mydef}
\vspace{-0.1cm}
It is shown in \cite{Abbaszadeh2013} that by using Cauchy- Schwarz inequality, if $\m f(\cdot)$ is Lipschitz continuous with constant $\gamma_l$, then $\m f(\cdot)$ is also OSL with $\gamma_s = \gamma_l$, showing that Lipschitz continuity implies OSL. Nonetheless, the converse is not true, as OSL is a one-sided nonlinearity as seen in \eqref{eq:osl_def}, while on the other hand
Lipschitz continuity is a two-sided nonlinearity. 
Therefore, there are some functions that are OSL but not Lipschitz continuous. For example, the scalar function $f(x) = -\mathrm{sgn}(x)\sqrt{x}$ for all $x\in\mbb{R}$ is OSL with $\gamma_{s} = 0$ but not Lipschitz continuous \cite{Abbaszadeh2013}. Readers are referred to \cite{Abbaszadeh2013} for a more detailed discussion and examples on the relation between Lipschitz continuity and OSL.

\subsection{Lipschitz Continuity and QIB: A New Understanding}\label{ssec:lip_qib_relation}
The QIB condition is introduced in the control theoretic literature (for the first time as far as we know) in \cite{Abbaszadeh2010,Abbaszadeh2013} and is extensively used alongside with OSL for observer design purposes.  If $\m f(\cdot)$ is Lipschitz continuous with Lipschitz constant $\gamma_l$, then it follows directly from \eqref{eq:qib_def} that $\m f(\cdot)$ is also QIB with $\gamma_{q1} = \gamma_l^2$ and $\gamma_{q2} = 0$ \cite{Abbaszadeh2013}. 

The question arises if a function satisfying the QIB condition is also necessarily Lipschitz continuous. It is claimed in \cite{Abbaszadeh2013} that this is not true in general and it is remarked  that if $\m f(\cdot)$ is OSL and QIB with $\gamma_{q2}$ positive, then $\m f(\cdot)$ is Lipschitz continuous.
 
On that regard, a further investigation shows that if a function $\m f(\cdot)$ is QIB then it \emph{is necessarily} Lipschitz continuous. This is shown in the following theorem.




\vspace{-0.1cm}
\begin{theorem}\label{thm:new_qib_lip}
Suppose that $\m f(\cdot)$  is QIB with constants $\gamma_{q1},\gamma_{q2}\in\mbb{R}$. Then $\m f(\cdot)$ is also Lipschitz continuous, where  $\gamma_{q1}, \gamma_{q2}$ necessarily satisfy the inequality
\begin{align}
2 \gamma_{q1}+|\gamma_{q2}|^2 \geq 0. \label{eq:qib_constants_condition}
\end{align}
\end{theorem}
\begin{proof}
First, recall Cauchy's inequality for real numbers $a,b\in\mbb{R}_{+}$, which is given as $ab \leq \tfrac{1}{2}a^2 + \tfrac{1}{2}b^2$. Suppose now that $\m f(\cdot)$ is QIB with constants $\gamma_{q1},\gamma_{q2}\in\mbb{R}$. Using Cauchy-Schwarz inequality and Cauchy's inequality mentioned previously, for $\m G = \m I$ we obtain the following results
	\begin{align}
	&\gamma_{q2}\langle\m f(\m x,\m u)-\m f(\hat{\m x},\m u),\m x-\hat{\m x}\rangle  \nonumber \\
	&\qquad \qquad \leq \abs {\gamma_{q2}\langle\m f(\m x,\m u)-\m f(\hat{\m x},\m u),\m x-\hat{\m x}\rangle} \nonumber \\
	&\qquad \qquad \leq \norm{\m f(\m x,\m u)-\m f(\hat{\m x},\m u)}_2\abs {\gamma_{q2}}\norm{\m x-\hat{\m x}}_2 \nonumber \\
	&\qquad \qquad  \leq\tfrac{1}{2}\norm{\m f(\m x,\m u)-\m f(\hat{\m x},\m u)}_2^2 + \tfrac{1}{2}\abs {\gamma_{q2}}^2\norm{\m x-\hat{\m x}}_2^2.\nonumber 
	\end{align}
Since $\m f(\cdot)$ is QIB, then from the above inequality and \eqref{eq:qib_def}, one can easily verify that
\begin{align}
&\norm {\m f(\m x,\m u)-\m f(\hat{\m x},\m u)}_2^{2}\nonumber \\ &\quad \leq\gamma_{q1} \norm{\m x - \hat{\m x}}_2^{2}+\gamma_{q2}\langle\m f(\m x,\m u)-\m f(\hat{\m x},\m u),\m x-\hat{\m x}\rangle\nonumber \\ 
&\quad\leq\gamma_{q1} \norm{\m x - \hat{\m x}}_2^{2} +\tfrac{1}{2}\norm{\m f(\m x,\m u)-\m f(\hat{\m x},\m u)}_2^2\nonumber \\  &\quad\quad+ \tfrac{1}{2}\abs {\gamma_{q2}}^2\norm{\m x-\hat{\m x}}_2^2,\nonumber
\end{align}
from which we can imply 
\begin{align}
\norm {\m f(\m x,\m u)-\m f(\hat{\m x},\m u)}_2^2 \leq \left(2\gamma_{q1} + \abs {\gamma_{q2}}^2\right)\norm{\m x-\hat{\m x}}_2^2.\label{eq:new_qib_lip}
\end{align}
Since $\mathbfcal{X}$ has nonempty interior, we can consider two points $(\m x,\m u)$ and $(\hat{\m x},\m u)$ such that $\m x\neq \hat{\m x}$. From \eqref{eq:new_qib_lip}, then it follows that
\begin{align*}
 \frac{\norm {\m f(\m x,\m u)-\m f(\hat{\m x},\m u)}_2^2}{\norm{\m x-\hat{\m x}}_2^2} \leq \left(2\gamma_{q1} + \abs {\gamma_{q2}}^2\right), 
\end{align*}
from which we deduce \eqref{eq:qib_constants_condition} by noting that the left-hand side is nonnegative. \eqref{eq:new_qib_lip} then implies
\begin{align}
\norm {\m f(\m x,\m u)-\m f(\hat{\m x},\m u)}_2 \leq \sqrt{2\gamma_{q1} + \abs {\gamma_{q2}}^2}\norm{\m x-\hat{\m x}}_2.
\end{align}
This completes the proof.
\end{proof}
\vspace{-0.1cm}

\vspace{-0.1cm}
	The above result shows that QIB implies Lipschitz continuity. Moreover since Lipschitz continuity implies QIB, we conclude that the class of  Lipschitz functions and QIB functions are the same. Moreover, the corresponding Lipschitz constant can be computed as $\gamma_l = \sqrt{2\gamma_{q1} + \abs {\gamma_{q2}}^2}$ from the QIB constants $\gamma_{q1}, \gamma_{q2}$. 
\begin{myrem}\label{rem:2}
	Despite QIB and Lipschitz continuity characterizing the same function sets, QIB can still be beneficial for less conservative observer design purpose. This is due to the fact that \textit{(a)} OSL is more general than Lipschitz and usually it is paired with QIB in observer design purpose and \textit{(b)} the QIB constants $\gamma_{q1,q2}$ can still be negative provided that \eqref{eq:qib_constants_condition} holds.
\end{myrem}

In the ensuing section, we shift our attention to NDS parameterization for OSL and QIB by posing the problem of bounding these constants as global maximization problems.  

\vspace{-0.0cm}
\section{NDS Parameterization\vspace{-0.05cm}}\label{sec:nonlinear_classification}
This section discusses our approach for parameterizing $\m f(\cdot)$ from NDS \eqref{eq:gen_dynamic_systems} into OSL and QIB as global maximization problems, in which the objective functions are given in closed-form expressions.

\vspace{-0.cm}
\subsection{One-Sided Lipschitz Parameterization}\label{ssec:osl}
Here we derive some numerical methods to compute OSL constant $\gamma_s$. To that end, first we propose numerical formulations that provide lower and upper bounds towards the left-hand side of one-sided Lipschitz condition given in \eqref{eq:osl_def}. 
\vspace{-0.1cm}
\begin{myprs}\label{prs:osl_bounds}
	\reducewordspace
	For the nonlinear function $\m f :\mathbb{R}^n\times \mathbb{R}^m\rightarrow \mathbb{R}^g$ in \eqref{eq:gen_dynamic_systems}, there exist $\bar{\gamma},\barbelow{\gamma}\in\mathbb{R}$ such that for $(\m x, \m u), (\hat{\m x}, \m u)\in \mathbf{\Omega}$ 
	\begin{subequations}\label{eq:one_sided_Lipschitz_bounds_theorem}
		\begin{align}
		\barbelow{\gamma} \norm{\m x - \hat{\m x}}_2^{2}\leq \langle\m G(\m f(\m x,\m u)-\m f(\hat{\m x},\m u)),\m x-\hat{\m x}\rangle\leq \bar{\gamma} \norm{\m x - \hat{\m x}}_2^{2}, \label{eq:one_sided_Lipschitz}
		\end{align}
		where $\bar{\gamma}$ and $\barbelow{\gamma}$ are given as
		\begin{align}
		\bar{\gamma} &= \max_{(\m{x},\m u)\in \mathbf{\Omega}}\lambda_{\mathrm{max}}\left(\frac{1}{2}\left(\mathbf{\Xi} (\m x,\m u) + \mathbf{\Xi}^{\top} (\m x,\m u)\right)\right) \label{eq:one_sided_Lipschitz_upper}\\
		\barbelow{\gamma} &= \min_{(\m{x},\m u)\in \mathbf{\Omega}}\lambda_{\mathrm{min}}\left(\frac{1}{2}\left(\mathbf{\Xi}(\m x,\m u) + \mathbf{\Xi}^{\top} (\m x,\m u)\right)\right),\label{eq:one_sided_Lipschitz_lower}
		\end{align}
		where each of the $i$-th and $j$-th element of $\mathbf{\Xi} (\cdot)$ is specified as
		\begin{align}
		{\Xi}_{(i,j)} (\m x,\m u) := \sum_{k\in \mbb{I}(g)}\hspace{-0.1cm}G_{(i,k)}\dfrac{\partial f_k}{\partial x_j}(\m x,\m u).\label{eq:one_sided_Lipschitz_matrix}
		\end{align}
	\end{subequations}
\end{myprs}
The proof of Proposition \ref{prs:osl_bounds} is available in \cite{nugroho2019tac}.
From this proposition, we deduce that $\gamma_s = \bar{\gamma}$. This result generalizes the approach to compute one-sided Lipschitz constant mentioned in \cite{Abbaszadeh2010} in two ways. Firstly, our result applies for a more general form of NDS expressed in \eqref{eq:gen_dynamic_systems} and secondly, we give a lower bound for the left-hand side of OSL condition presented in \eqref{eq:osl_def}, which can be useful for determining quadratically inner-bounded constants, as we see later in the next section. However, the non closed-form expression for $\bar{\gamma}$ described in \eqref{eq:one_sided_Lipschitz_upper} and similarly in \cite{Abbaszadeh2010} makes it difficult to compute $\gamma_s$ via deterministic global optimization methods. Motivated by this limitation, we present several solutions to this problem, first of which is presented below---see \cite{nugroho2019tac} for the proof.
\vspace{-0.1cm}
\begin{theorem}\label{cor:osl_bounds}
	\reducewordspace
	The nonlinear function $\m f :\mathbb{R}^n\times \mathbb{R}^m\rightarrow \mathbb{R}^g$ in \eqref{eq:gen_dynamic_systems} is one-sided Lipschitz continuous in $\mathbf{\Omega}$ satisfying
	\begin{subequations} 
		\begin{align}
	{	\vphantom{\left(\frac{v_f}{l}\right)} \langle\m G(\m f(\m x,\m u)-\m f(\hat{\m x},\m u)),\m x-\hat{\m x}\rangle\leq \gamma_s \norm{\m x - \hat{\m x}}_2^{2}}
		\end{align}
		for all $(\m x, \m u), (\hat{\m x}, \m u)\in \mathbf{\Omega}$ with
		\begin{align}\label{eq:osl_bound_for_gamma_s}
\hspace{-0.4cm}	{	\gamma_s = \left(\max_{(\m x, \m u)\in \m \Omega} \sum_{i, j\in \mbb{I}(g)} \abs{\sum_{k\in \mbb{I}(g)} G_{(i,k)} \dfrac{\partial f_k}{\partial x_j}(\m x, \m u)}^2\hspace{0.05cm}\right)^{\hspace{-0.05cm}\!1/2}.}
		\end{align}
	\end{subequations}
\end{theorem}



Another less straightforward approach---that is potentially less conservative approach to compute $\gamma_s$ than \eqref{eq:osl_bound_for_gamma_s}---is to make use of the \textit{Gershgorin's circle theorem}.
The next proposition recapitulates this approach to compute an upper bound for the greatest eigenvalue of a symmetric matrix.
\vspace{-0.1cm}
\begin{myprs}[From \cite{barany2017gershgorin}]\label{prs:gershgorin_max_eigenvalue}
	For any symmetric matrix $\m \Psi \in \mbb{S}^n$, the following inequality holds
	\begin{align}
	\lambda_{\mathrm{max}}\left(\m \Psi \right) \leq \max_{i\in \mbb{I}(n)}\left(\Psi_{(i,i)} + \sum_{j \in\mbb{I}(n)\setminus i}\abs{\Psi_{(i,j)}}\right).\label{prs:gershgorin_inequality_max_eigenvalue}
	\end{align}
\end{myprs}
Proposition \ref{prs:gershgorin_max_eigenvalue} provides an amenable way which can be used for computing one-sided Lipschitz constant $\gamma_s$ provided that $\m \Psi = \frac{1}{2}\left(\mathbf{\Xi} (\m x,\m u) + \mathbf{\Xi}^{\top} (\m x,\m u)\right)$. That is, we can consider that
\begin{align}
{\gamma_s =  \max_{i\in \mbb{I}(n)}\left(\max_{(\m{x},\m u)\in \mathbf{\Omega}}\left(\Psi_{(i,i)} + \sum_{j \in\mbb{I}(n)\setminus i}\abs{\Psi_{(i,j)}}\right)\hspace{-0.1cm}\right).\label{eq:osl_gershgorin}}
\end{align}
Note that in \eqref{prs:gershgorin_inequality_max_eigenvalue} it is possible for $\lambda_{\mathrm{max}}\left(\m \Psi \right)$ to be nonpositive assuming that $\m \Psi $ is diagonally dominant with $\Psi_{(i,i)}\leq 0$ for each $i\in\mbb{I}(n)$.  
In the following theorem, we develop another upper bound for the greatest eigenvalue of any symmetric matrix, which has different form compared to  \eqref{prs:gershgorin_inequality_max_eigenvalue}.
\vspace{-0.1cm}
\begin{theorem}\label{thm:vu_max_eigenvalue}
	It holds for any $\m \Psi \in \mbb{S}^n$ that
	\begin{subequations}
		\begin{align}
		\lambda_{\mathrm{max}}\left(\m \Psi \right) \leq \max_{i\in \mbb{I}(n)}\left(\Psi_{(i,i)} +\zeta_n\max_{j\in \mbb{I}(n)\setminus i}\abs{\Psi_{(i,j)}}\right),\label{thm:vu_inequality_max_eigenvalue}
		\end{align}	
		where $\zeta_n \in \mbb{R}_{++}$ is a scalar that depends on the dimension $n$ and is the optimal value of the following maximization problem
		\begin{align}
		\zeta_n = \hspace*{-0.0cm}\max_{\m v\in\mbb{R}^n}\;\; & \frac{1}{v_i}-1 \label{eq:vu_inequality_max_eigenvalue_const_1}\\
		\subjectto  \;\;	 & \sum_{j \in\mbb{I}(n)} \abs{v_i} = 1,\label{eq:vu_inequality_max_eigenvalue_const_2}\\
		&v_i > 0,\;\abs{ v_j}  \leq v_i,\;\forall j\in \mbb{I}(n)\setminus i.\label{eq:vu_inequality_max_eigenvalue_const_3}
		\end{align}
	\end{subequations}
\end{theorem}
In using Theorem \ref{thm:vu_max_eigenvalue}, one need to compute $\zeta_n$ beforehand, which essentially includes solving a nonconvex problem given in \eqref{eq:vu_inequality_max_eigenvalue_const_1}--\eqref{eq:vu_inequality_max_eigenvalue_const_3}. Indeed, this problem can be posed as a convex one, which is summarized next.
\vspace{-0.1cm}
\begin{myprs}\label{prs:computing_zeta}
The nonconvex maximization problem given in \eqref{eq:vu_inequality_max_eigenvalue_const_1}--\eqref{eq:vu_inequality_max_eigenvalue_const_3} is equivalent to the following convex problem
\begin{subequations}\label{eq:computing_zeta}
\begin{align}
v_i^* = \hspace*{-0.0cm}\min_{\m v,\m w\in\mbb{R}^n}\;\; & v_i \label{eq:computing_zeta_1}\\
\subjectto  \;\;	 & \sum_{j \in\mbb{I}(n)} w_i = 1, \label{eq:computing_zeta_2}\\
&v_i > 0,\;w_j  \leq v_i,\;\forall j\in \mbb{I}(n)\setminus i.\label{eq:computing_zeta_3}\\
& \abs{v_i} = w_i, \forall i\in \mbb{I}(n),\label{eq:computing_zeta_4}
\end{align}
where $\zeta_n$ in \eqref{eq:vu_inequality_max_eigenvalue_const_1} can be recovered from $\zeta_n^* = \dfrac{1}{v_i^*}-1$.
\end{subequations}
\end{myprs}
Notice that problem described in \eqref{eq:computing_zeta} can be solved using any convex programming solvers. In fact, and after solving it for various values of $n$, we observe that the optimal solution is $\zeta_n^* = n - 1$. Readers are referred to \cite{nugroho2019tac} for the proofs.

\vspace{-0.2cm}

\subsection{Quadratic Inner-Boundedness Parameterization }\label{ssec:qib}


The following theorem summarizes our result for computing constants that characterize QIB function sets through posing the problem again as a maximization problem.
\vspace{-0.1cm}
\begin{theorem}\label{prs:quadratic_inner_bounded}
	The nonlinear function $\m f :\mathbb{R}^n\times \mathbb{R}^m\rightarrow \mathbb{R}^g$ in \eqref{eq:gen_dynamic_systems} is locally quadratically inner-bounded in $\mathbf{\Omega}$ such that for any $(\m x, \m u), (\hat{\m x}, \m u)\in \mathbf{\Omega}$ the following holds
	\begin{subequations}\label{eq:quadratic_inner_bounded_theorem}
		\begin{align}
		&\langle\m G(\m f(\m x,\m u)-\m f(\hat{\m x},\m u)),\m G(\m f(\m x,\m u)-\m f(\hat{\m x},\m u))\rangle\leq \nonumber \\
		& \;\gamma_{q1} \norm{\m x - \hat{\m x}}_2^{2}+\gamma_{q2}\langle\m G(\m f(\m x,\m u)-\m f(\hat{\m x},\m u)),\m x-\hat{\m x}\rangle, \label{eq:quadratic_inner_bounded}
		\end{align}
		where for $\epsilon_1,\epsilon_2\in \mbb{R}_{+}$, $\gamma_{q2} = \epsilon_2-\epsilon_1$ and $\gamma_{q1}$ is specified as
		\begin{align}
	{	\gamma_{q1} = \hspace{-0.0cm}\epsilon_1\bar{\gamma}-\epsilon_2\barbelow{\gamma}+\max_{(\m{x},\m u)\in  \mathbf{\Omega}}\hspace{-0.0cm}\sum_{i\in\mbb{I}(n)}\hspace{-0.0cm}\norm{\nabla_{\hspace{-0.05cm}x} \hspace{0.04cm}\xi_i(\m x,\m u)}_2^2,\hspace{-0.00cm} \label{eq:gamma_q1_qib}}
		\end{align}
		where $\xi_i(\m x,\m u) := \sum_{j\in\mbb{I}(g)} G_{(i,j)}f_j(\m x, \m u)$ for $i\in\mbb{I}(n)$ and
		$\bar{\gamma}$ and $\barbelow{\gamma}$ are the optimal values of \eqref{eq:one_sided_Lipschitz_upper} and \eqref{eq:one_sided_Lipschitz_lower}.
	\end{subequations}
\end{theorem}

\vspace{-0.1cm}
This result---see \cite{nugroho2019tac} for the proof---allows the quadratically inner-bounded constants $\gamma_{q1}$ and $\gamma_{q2}$ to be parameterized with non-negative variables $\epsilon_1$ and $\epsilon_2$, hence giving a degree of freedom that can be useful for observer/controller design. 
\section{Computing OSL and QIB Constants}\label{sec:interval_optimization}
The previous section posed the problem of obtaining OSL and QIB constants as global maximization problems over a constrained set $\m \Omega$. That is, all of the boxed equations given in~\eqref{eq:osl_bound_for_gamma_s},~\eqref{eq:osl_gershgorin}, and~\eqref{eq:gamma_q1_qib} are global maximization problems, constrained by $(\m x, \m u) \in \m \Omega$.

Assuming that $\m \Omega$ is a convex set that includes, for example, upper and lower bounds on all states and control inputs (a valid assumption for an array of dynamic networks modeling infrastructure), all of these maximization problems are most likely nonconvex/nonconcave seeing the gradient of $\m f(\cdot)$ is analytically derived within the optimization. With that in mind, and for a specific nonlinear system, the gradient could potentially produce a concave cost function thereby making the global optimization tractable, as it can be solved via convex programming technique such as interior-point methods \cite{nemirovski2008interior}. 
Considering the more plausible scenario that the cost functions are mostly nonconcave, few approaches can be investigated to solve and compute the constants. We summarize these approaches here, while keeping in mind that this is outside the scope of this paper. The first approach to solve~\eqref{eq:osl_bound_for_gamma_s},~\eqref{eq:osl_gershgorin}, and~\eqref{eq:gamma_q1_qib} is through deterministic, global optimization methods such as inner and outer approximation \cite{tuy1998convex}, cutting-plane \cite{floudas2001encyclopedia}, and branch-and-bound methods \cite{Fowkes2013}. For a large scale system, this approach might prove to be intractable.

The second approach is manifested through using randomized sampling through generating a large number of low discrepancy sequences (LDS) and samples for $(\m x, \m u)$ inside $\m \Omega$; see \cite{Nugroho2019} for an example. This is then followed by evaluating~\eqref{eq:osl_bound_for_gamma_s},~\eqref{eq:osl_gershgorin}, and~\eqref{eq:gamma_q1_qib} and subsequently computing the values for constants $\gamma_{s,q1,q2}$ that maximize the corresponding objective functions inside $\m \Omega$.
The third approach is to use interval-based optimization methods. This approach gained momentum recently, through studies that showcased the potential of interval arithmetic in dealing with nonconvex optimization problems; see the studies \cite{van2004termination,markot2006new,Ratschek1991}. We have pursued this particular approach and developed it further for vector valued global maximization in our recent work \cite{nugroho2019tac}. Numerical tests include some preliminary results for this approach, which are presented next.

\vspace{-0.0cm}

\section{An Illustrative Example}\label{sec:numerical_tests}
We consider the dynamics a moving object in $2-$D plane, described in
\cite{Abbaszadeh2010,Abbaszadeh2013} by 
\begin{align}\label{eq:state_space_mov_obj_simple}
\hspace{-0.4cm}\dot{\m x} &= \bmat{1&-1\\ 1& 1}\m x + \bmat{-x_1(x_1^2+x_2^2) \\ -x_2(x_1^2+x_2^2)},\;\;
 y  = \bmat{0 & 1}\m x.
\end{align}
Notice that the nonlinear mapping $\m f:\mbb{R}^2\rightarrow\mbb{R}^2$ of \eqref{eq:state_space_mov_obj_simple} can be also expressed as $\m f(\m x) = -\m x\norm{\m x}^2_2$. 

\subsection{Addressing Erroneous Analytical Parameterization}\label{ssec:correction}

\vspace{-0.15cm}

Unfortunately, the numerical section of \cite{Abbaszadeh2013,zhang2012full} contains some mistakes. The study \cite{Abbaszadeh2013} considers the system described in \eqref{eq:state_space_mov_obj_simple} with domain of interest $\m \Omega = \left\{\m x\in \mbb{R}^2\,|\, \norm {\m x}_2 \leq r\right\}$ where $r \geq 0$. Albeit it is true that such $\m f(\cdot)$ is globally OSL with $\gamma_s = 0$, the claim that $\m f(\cdot)$ is locally QIB in $\m \Omega$ given 
\begin{align}\label{eq:qib_wrong}
\begin{split}
&r = \min \left(\sqrt{-\tfrac{\gamma_{q2}}{4}},\sqrt[4]{\gamma_{q1}+\tfrac{\gamma_{q2}^2}{4}}\right),\; \gamma_{q2}\leq 0,\;\\ &\gamma_{q1}+\tfrac{\gamma_{q2}^2}{4} > 0,
\end{split}
\end{align} 
is incorrect. More specifically, \cite{Abbaszadeh2013} consider the values $\gamma_{q1} = -200$  and $\gamma_{q2} = -141$, which according to \eqref{eq:qib_wrong} gives $r = 5.9372$. Now take the two points $\m x = [1\;\,0]^\top$ and $\hat{  \m x}= \m 0$. One can verify that $\m x,\hat{  \m x}\in\m \Omega$. The left-hand side of \eqref{eq:qib_def} is equal to
\begin{align*}
\norm {\m f(\m x)-\m f(\hat{  \m x})}_2^2 = \norm {\m f(\m x)}_2^2 = \norm {-\m x\norm{\m x}^2_2}_2^2 = 1.
\end{align*}  
If $\m f$ were QIB, we would get
\begin{align*}
1 &\leq \gamma_{q1} \norm{\m x - \hat{\m x}}_2^{2} +\gamma_{q2}\langle\m f(\m x)-\m f(\hat{\m x}),\m x-\hat{\m x}\rangle \nonumber \\
&=-200 \norm{\m x}_2^{2} + (-141)\langle-\m x\norm{\m x}^2_2,\m x\rangle = -59,
\end{align*}
a contradiction. This shows that the conditions \eqref{eq:qib_wrong} do not imply that $\m f(\cdot)$ is QIB. Unfortunately, the same condition is also used in several works on observer design procedure for NDS with OSL and QIB nonlinearities; see for example \cite{zhang2012full,Zhang2012iet,HUANG20173305,ahmad2015observer,AHMAD2016903}. 
For instance, in \cite{zhang2012full}, the values $\gamma_{q1} = -99$  and $\gamma_{q2} = -100$ are chosen, giving $r = 5$. For $\m x = [0.5\;\,0]^\top$ and $\hat{  \m x}= \m 0$, one can get a similar contradiction. 

A further analysis reveals that it is necessary to have $\gamma_{q1} \geq 0$, given the same set $\m \Omega$.  
To proceed, take $\m x = [a\;\,0]^\top$ and $\hat{  \m x}= \m 0$ with $0 <|a| \leq r$ and then evaluate the following expressions related to QIB condition
\begin{subequations} \label{eq:qib_example}
	\begin{align}
	&\norm {\m f(\m x)-\m f(\hat{  \m x})}_2^2 = \norm {\m f(\m x))}_2^2 = \norm {-\m x\norm{\m x}^2_2}_2^2 = a^6 \label{eq:qib_example_1}\\
	&\gamma_{q1} \norm{\m x - \hat{\m x}}_2^{2} +\gamma_{q2}\langle\m f(\m x)-\m f(\hat{\m x}),\m x-\hat{\m x}\rangle \nonumber \\
	&\qquad = \gamma_{q1} \norm{\m x }_2^{2}  +\gamma_{q2}\langle\m f(\m x),\m x\rangle  = a^2\gamma_{q1}-a^4\gamma_{q2}.  \label{eq:qib_example_2}
	\end{align}
\end{subequations}
If $\m f(\cdot)$ is QIB, then it follows from \eqref{eq:qib_example} that
\begin{align*}
a^6 \leq  a^2\gamma_{q1}-a^4\gamma_{q2}.
\end{align*}
We can then divide the above by $a^2$ to deduce
\begin{align*}
a^4 \leq \gamma_{q1}-a^2\gamma_{q2},
\end{align*}
and take the limit $a\to 0$ on both sides to get $0 \leq \gamma_{q1}$, finishing the proof of our claim. This means that in particular 
$2\gamma_{q1} +\abs{\gamma_{q2}}^2 \geq 0,$
and therefore this example is compatible with the statement of Theorem \ref{thm:new_qib_lip}.

Correct conditions to ensure $\m f(\cdot)$ is QIB are as follows:
\begin{align}
2r^2 &\leq -\tfrac{\gamma_{q2}}{2},\;\;r^4 \leq \gamma_{q1} -{\gamma_{q2}}r^2,\;\;\gamma_{q1} \geq 0,\;\; \gamma_{q2} < 0. \label{eq:sufficient}
\end{align}
For example, $\gamma_{q1} = 100$, $\gamma_{q2} = -100$, and $r = 5$ satisfy \eqref{eq:sufficient} and hence are sufficient to ensure that $\m f(\cdot)$ is QIB in $\m \Omega$. As a result, further investigation is required to correctly evaluate the observer designs developed in \cite{Abbaszadeh2013,zhang2012full,HUANG20173305,ahmad2015observer,AHMAD2016903,Zhang2012iet} using the revised QIB constants computed from \eqref{eq:sufficient}. 

\subsection{Numerical Parameterization and State Estimation}\label{ssec:simulation}
This section showcases the proposed methodologies for parameterizing a nonlinear dynamics of a moving object into OSL and QIB. All simulations are performed using {MATLAB} R2017b.
YALMIP's \cite{Lofberg2004} optimization package together with MOSEK's \cite{andersen2000mosek} solver
are used to solve the SDP.
 
In this example, we use $\m \Omega = \left\{\m x\in \mbb{R}^2\,|\, x_i\in [-r,r]\right\}$ where $r = 5$. This numerical example focuses on \textit{(a)} finding both OSL constant $\gamma_s$ and QIB constants $\gamma_{q1}$ and $\gamma_{q2}$ that are \textit{(b)} useful for observer design of the nonlinear system~\eqref{eq:state_space_mov_obj_simple}. To that end, first we attempt to find $\gamma_s$ using the method given in Proposition \ref{prs:gershgorin_max_eigenvalue}.  Matrix $\m \Xi (\m x)$ for this nonlinearity is given as
\begin{align*}
\m \Xi (\m x) = \bmat{-3x_1^2-x_2^2 & -2x_1x_2 \\ -2x_1x_2 & -x_1^2-3x_2^2}.
\end{align*} 
Since $\m \Xi (\m x)$  in the above example is already symmetric, then according to  \eqref{prs:gershgorin_inequality_max_eigenvalue}, $\gamma_s$ can be computed from using \eqref{eq:osl_gershgorin}.
The \textit{interval-based algorithm}, which uses interval arithmetic and branch-and-bound (BnB) routines, computes the tightest interval containing upper and lower bounds for the the objective function of \eqref{eq:osl_gershgorin}. In particular, this algorithm splits the set $\m \Omega$ into smaller subsets and removes the ones that positively do not contain any maximizers. 
Upon running this algorithm, we find that the value of $\gamma_s$ is decreasing to a value near zero, showing that the OSL constant for this system is zero. This results corroborates the fact that, as proven in \cite{Abbaszadeh2010}, the system is indeed globally one-sided Lipschitz with an analytical OSL constant $\gamma_s = 0$. 

Next, we determine the quadratically inner-bounded constant by solving the following problems using interval-based algorithm
	\begin{align*}
	\barbelow{\gamma} &= \min_{i\in\{1,2\}}\left(\min_{\m{x}\in \mathbf{\Omega}}\left(\m \Xi_{(i,i)} -\sum_{j \in\{1,2\}\setminus i}\abs{\m \Xi_{(i,j)}}\right)\hspace{-0.1cm}\right) \\ 
	\gamma_m &= \max_{\m{x}\in \mathbf{\Omega}}\sum_{i\in \{1,2\}}\norm{\nabla_{\hspace{-0.05cm}x} f_i(\m{x})}_2^2. 
	\end{align*}
Solving the above optimization, we obtain $\barbelow{\gamma} = -150$ and $\gamma_m = 2.5\times10^4$. According to Proposition \ref{prs:quadratic_inner_bounded}, for any $\epsilon_1,\epsilon_2\in \mbb{R}_{+}$, the quadratically inner-bounded constant can be constructed as $\gamma_{q2} = \epsilon_2-\epsilon_1$ and $\gamma_{q1} = \epsilon_1\gamma_s -\epsilon_2\barbelow{\gamma}+\gamma_m$. By setting the constants $\epsilon_1 = 10^5$ and $\epsilon_2 = 10^{-1}$, we obtain $\gamma_{q1} = 25015$ and $\gamma_{q2} = -9999.89$. To test the applicability of the computed one-sided Lipschitz and quadratically inner-bounded constants for state estimation, we implement an observer developed in \cite{zhang2012full} using the computed constants. Given the solution to the SDP \eqref{eq:LMI_osl_qib_obs},  we are successfully able to obtain converging estimation error. Fig. \ref{Fig:mov_obj_1c} depicts the corresponding trajectories.

\begin{figure}[t]
	\centering
	\hspace{-0.3cm}
{\includegraphics[scale=0.4]{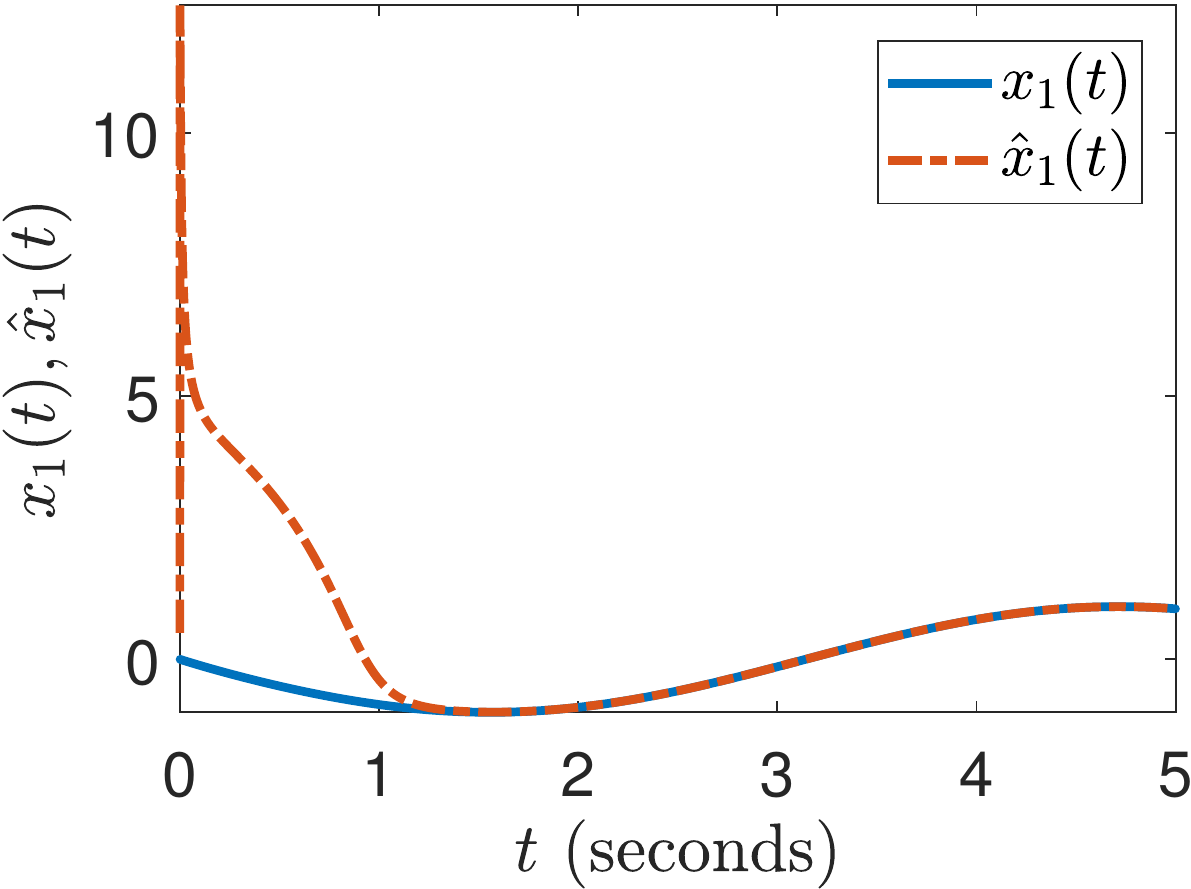}\label{Fig:mov_obj_1c}}{\vspace{-0.1cm}}\hspace{-0.0em}
	\vspace{-0.1cm}
	\caption{Numerical test results for dynamics of a moving object: trajectories of actual $x_1(t)$ and estimated $\hat{x}_1(t)$ state.}
	\vspace{-0.554cm}
\end{figure} 


\section{Conclusions and Acknowledgments}\label{sec:conclusion}
In this paper, we present a new understanding for the relation between Lipschitz continuous and QIB function sets. Our findings state that QIB implies Lipschitz continuous, which suggests that Lipschitz continuous and QIB share the same function sets. We also present numerical methods to compute the corresponding constants for OSL and QIB, posed as global maximization problems. Numerical results demonstrate the applicability of the proposed approach to compute these constants for practical observer designs.

Finally, we would like to acknowledge the editor and the four reviewers who provided thoughtful suggestions and constructive criticism. Specifically, a reviewer's comments motivated the developments of Section~\ref{ssec:correction} and we are grateful for that.

%
\bibliographystyle{IEEEtran}	\bibliography{IEEEabrv,bib_file}

\begin{thebibliography}{10}
\providecommand{\url}[1]{#1}
\csname url@samestyle\endcsname
\providecommand{\newblock}{\relax}
\providecommand{\bibinfo}[2]{#2}
\providecommand{\BIBentrySTDinterwordspacing}{\spaceskip=0pt\relax}
\providecommand{\BIBentryALTinterwordstretchfactor}{4}
\providecommand{\BIBentryALTinterwordspacing}{\spaceskip=\fontdimen2\font plus
\BIBentryALTinterwordstretchfactor\fontdimen3\font minus
  \fontdimen4\font\relax}
\providecommand{\BIBforeignlanguage}[2]{{%
\expandafter\ifx\csname l@#1\endcsname\relax
\typeout{** WARNING: IEEEtran.bst: No hyphenation pattern has been}%
\typeout{** loaded for the language `#1'. Using the pattern for}%
\typeout{** the default language instead.}%
\else
\language=\csname l@#1\endcsname
\fi
#2}}
\providecommand{\BIBdecl}{\relax}
\BIBdecl

\bibitem{boyd1994linear}
S.~Boyd, L.~El~Ghaoui, E.~Feron, and V.~Balakrishnan, \emph{Linear matrix
  inequalities in system and control theory}.\hskip 1em plus 0.5em minus
  0.4em\relax Siam, 1994, vol.~15.

\bibitem{vanantwerp2000tutorial}
J.~G. VanAntwerp and R.~D. Braatz, ``A tutorial on linear and bilinear matrix
  inequalities,'' \emph{Journal of process control}, vol.~10, no.~4, pp.
  363--385, 2000.

\bibitem{Phanomchoeng2010}
G.~{Phanomchoeng} and R.~{Rajamani}, ``Observer design for lipschitz nonlinear
  systems using riccati equations,'' in \emph{Proceedings of the 2010 American
  Control Conference}, June 2010, pp. 6060--6065.

\bibitem{ichalal2012observer}
D.~Ichalal, B.~Marx, S.~Mammar, D.~Maquin, and J.~Ragot, ``Observer for
  lipschitz nonlinear systems: mean value theorem and sector nonlinearity
  transformation,'' in \emph{2012 IEEE International Symposium on Intelligent
  Control}.\hskip 1em plus 0.5em minus 0.4em\relax IEEE, 2012, pp. 264--269.

\bibitem{zemouche2006observer}
A.~Zemouche and M.~Boutayeb, ``Observer design for lipschitz nonlinear systems:
  the discrete-time case,'' \emph{IEEE Transactions on Circuits and Systems II:
  Express Briefs}, vol.~53, no.~8, pp. 777--781, 2006.

\bibitem{Nugroho2019ITS}
S.~A. {Nugroho}, A.~F. {Taha}, and C.~G. {Claudel}, ``A control-theoretic
  approach for scalable and robust traffic density estimation using convex
  optimization,'' \emph{IEEE Transactions on Intelligent Transportation
  Systems}, pp. 1--15, 2019.

\bibitem{Nugroho2019TPS}
S.~A. {Nugroho}, A.~F. {Taha}, and J.~{Qi}, ``Robust dynamic state estimation
  of synchronous machines with asymptotic state estimation error performance
  guarantees,'' \emph{IEEE Transactions on Power Systems}, pp. 1--1, 2019.

\bibitem{Yadegar2018}
M.~Yadegar, A.~Afshar, and M.~Davoodi, ``Observer-based tracking controller
  design for a class of lipschitz nonlinear systems,'' \emph{Journal of
  Vibration and Control}, vol.~24, no.~11, pp. 2112--2119, 2018.

\bibitem{zemouche2017robust}
A.~Zemouche, R.~Rajamani, H.~Kheloufi, and F.~Bedouhene, ``Robust
  observer-based stabilization of lipschitz nonlinear uncertain systems via
  lmis-discussions and new design procedure,'' \emph{International Journal of
  Robust and Nonlinear Control}, vol.~27, no.~11, pp. 1915--1939, 2017.

\bibitem{Ekramian2017}
M.~Ekramian, ``Observer-based controller for lipschitz nonlinear systems,''
  \emph{International Journal of Systems Science}, vol.~48, no.~16, pp.
  3411--3418, 2017.

\bibitem{Abbaszadeh2010}
M.~Abbaszadeh and H.~J. Marquez, ``Nonlinear observer design for one-sided
  lipschitz systems,'' in \emph{Proceedings of the 2010 American Control
  Conference}, June 2010, pp. 5284--5289.

\bibitem{zhang2012full}
W.~Zhang, H.~Su, H.~Wang, and Z.~Han, ``Full-order and reduced-order observers
  for one-sided lipschitz nonlinear systems using riccati equations,''
  \emph{Communications in Nonlinear Science and Numerical Simulation}, vol.~17,
  no.~12, pp. 4968--4977, 2012.

\bibitem{benallouch2012observer}
M.~Benallouch, M.~Boutayeb, and M.~Zasadzinski, ``Observer design for one-sided
  lipschitz discrete-time systems,'' \emph{Systems \& Control Letters},
  vol.~61, no.~9, pp. 879--886, 2012.

\bibitem{zhang2015unknown}
W.~Zhang, H.~Su, F.~Zhu, and G.~M. Azar, ``Unknown input observer design for
  one-sided lipschitz nonlinear systems,'' \emph{Nonlinear Dynamics}, vol.~79,
  no.~2, pp. 1469--1479, 2015.

\bibitem{wu2015observer}
R.~Wu, W.~Zhang, F.~Song, Z.~Wu, and W.~Guo, ``Observer-based stabilization of
  one-sided lipschitz systems with application to flexible link manipulator,''
  \emph{Advances in Mechanical Engineering}, vol.~7, no.~12, p.
  1687814015619555, 2015.

\bibitem{Rastegari2019}
A.~Rastegari, M.~M. Arefi, and M.~H. Asemani, ``Robust $h_{\infty}$ sliding
  mode observer-based fault-tolerant control for one-sided lipschitz nonlinear
  systems,'' \emph{Asian Journal of Control}, vol.~21, no.~1, pp. 114--129,
  2019.

\bibitem{Gholami2019}
H.~Gholami and T.~Binazadeh, ``Observer-based $h_{\infty}$ finite-time
  controller for time-delay nonlinear one-sided lipschitz systems with
  exogenous disturbances,'' \emph{Journal of Vibration and Control}, vol.~25,
  no.~4, pp. 806--819, 2019.

\bibitem{SCHULZEDARUP2018135}
M.~S. Darup and M.~Mönnigmann, ``Fast computation of lipschitz constants on
  hyperrectangles using sparse codelists,'' \emph{Computers and Chemical
  Engineering}, vol. 116, pp. 135 -- 143, 2018.

\bibitem{Chakrabarty2019}
A.~{Chakrabarty}, D.~K. {Jha}, G.~T. {Buzzard}, Y.~{Wang}, and
  K.~{Vamvoudakis}, ``{Safe Approximate Dynamic Programming Via Kernelized
  Lipschitz Estimation},'' \emph{arXiv e-prints}, p. arXiv:1907.02151, Jul
  2019.

\bibitem{Fazlyab2019}
M.~Fazlyab, A.~Robey, H.~Hassani, M.~Morari, and G.~J. Pappas, ``Efficient and
  accurate estimation of lipschitz constants for deep neural networks,''
  \emph{CoRR}, vol. abs/1906.04893, 2019.

\bibitem{Nugroho2019}
S.~A. Nugroho, A.~F. Taha, and J.~Qi, ``Characterizing the nonlinearity of
  power system generator models,'' in \emph{2019 American Control Conference
  (ACC)}.\hskip 1em plus 0.5em minus 0.4em\relax IEEE, July 2019, pp.
  1936--1941.

\bibitem{Abbaszadeh2013}
\BIBentryALTinterwordspacing
M.~Abbaszadeh and H.~J. Marquez, ``Design of nonlinear state observers for
  one-sided lipschitz systems,'' \emph{CoRR}, vol. abs/1302.5867, 2013.
  [Online]. Available: \url{http://arxiv.org/abs/1302.5867}
\BIBentrySTDinterwordspacing

\bibitem{trench2003introduction}
W.~Trench, \emph{Introduction to Real Analysis}.\hskip 1em plus 0.5em minus
  0.4em\relax Prentice Hall/Pearson Education, 2003.

\bibitem{nugroho2019tac}
\BIBentryALTinterwordspacing
S.~A. Nugroho, A.~F. {Taha}, and V.~Hoang, ``Nonlinear dynamic systems
  parameterization using interval-based global optimization: Computing
  lipschitz constants and beyond,'' \emph{IEEE Transactions on Automatic
  Control}, September 2019, in review. [Online]. Available:
  \url{https://bit.ly/2kgqgdE}
\BIBentrySTDinterwordspacing

\bibitem{barany2017gershgorin}
I.~B{\'a}r{\'a}ny and J.~Solymosi, ``Gershgorin disks for multiple eigenvalues
  of non-negative matrices,'' in \emph{A Journey Through Discrete
  Mathematics}.\hskip 1em plus 0.5em minus 0.4em\relax Springer, 2017, pp.
  123--133.

\bibitem{nemirovski2008interior}
A.~S. Nemirovski and M.~J. Todd, ``Interior-point methods for optimization,''
  \emph{Acta Numerica}, vol.~17, pp. 191--234, 2008.

\bibitem{tuy1998convex}
H.~Tuy, \emph{Convex Analysis and Global Optimization}, ser. Advances in
  Natural and Technological Hazards Research.\hskip 1em plus 0.5em minus
  0.4em\relax Springer US, 1998.

\bibitem{floudas2001encyclopedia}
C.~A. Floudas and P.~M. Pardalos, \emph{Encyclopedia of optimization}.\hskip
  1em plus 0.5em minus 0.4em\relax Springer Science \& Business Media, 2001,
  vol.~1.

\bibitem{Fowkes2013}
J.~M. Fowkes, N.~I.~M. Gould, and C.~L. Farmer, ``A branch and bound algorithm
  for the global optimization of hessian lipschitz continuous functions,''
  \emph{Journal of Global Optimization}, vol.~56, no.~4, pp. 1791--1815, Aug
  2013.

\bibitem{van2004termination}
M.~Van~Emden and B.~Moa, ``Termination criteria in the moore-skelboe algorithm
  for global optimization by interval arithmetic,'' in \emph{Frontiers in
  Global Optimization}.\hskip 1em plus 0.5em minus 0.4em\relax Springer, 2004,
  pp. 585--597.

\bibitem{markot2006new}
M.~C. Mark{\'o}t, J.~Fern{\'a}ndez, L.~G. Casado, and T.~Csendes, ``New
  interval methods for constrained global optimization,'' \emph{Mathematical
  Programming}, vol. 106, no.~2, pp. 287--318, 2006.

\bibitem{Ratschek1991}
\BIBentryALTinterwordspacing
H.~Ratschek and R.~L. Voller, ``What can interval analysis do for global
  optimization?'' \emph{Journal of Global Optimization}, vol.~1, no.~2, pp.
  111--130, Jun 1991. [Online]. Available:
  \url{https://doi.org/10.1007/BF00119986}
\BIBentrySTDinterwordspacing

\bibitem{Zhang2012iet}
W.~Zhang, ``\BIBforeignlanguage{English}{Non-linear observer design for
  one-sided lipschitz systems: an linear matrix inequality approach},''
  \emph{\BIBforeignlanguage{English}{IET Control Theory \& Applications}},
  vol.~6, pp. 1297--1303(6), June 2012.

\bibitem{HUANG20173305}
J.~Huang, W.~Zhang, M.~Shi, L.~Chen, and L.~Yu, ``{$H_{\infty}$} observer
  design for singular one-sided lur’e differential inclusion system,''
  \emph{Journal of the Franklin Institute}, vol. 354, no.~8, pp. 3305 -- 3321,
  2017.

\bibitem{ahmad2015observer}
S.~Ahmad, R.~Majeed, K.-S. Hong, and M.~Rehan, ``Observer design for one-sided
  lipschitz nonlinear systems subject to measurement delays,''
  \emph{Mathematical Problems in Engineering}, vol. 2015, pp. 1 -- 13, 2015.

\bibitem{AHMAD2016903}
S.~Ahmad and M.~Rehan, ``On observer-based control of one-sided lipschitz
  systems,'' \emph{Journal of the Franklin Institute}, vol. 353, no.~4, pp. 903
  -- 916, 2016.

\bibitem{Lofberg2004}
J.~L\"{o}fberg, ``{YALMIP}: A toolbox for modeling and optimization in
  matlab,'' in \emph{Proc. IEEE Int. Symp. Computer Aided Control Syst.
  Design}, 2004, pp. 284--289.

\bibitem{andersen2000mosek}
E.~D. Andersen and K.~D. Andersen, ``The mosek interior point optimizer for
  linear programming: an implementation of the homogeneous algorithm,'' in
  \emph{High performance optimization}.\hskip 1em plus 0.5em minus 0.4em\relax
  Springer, 2000, pp. 197--232.

\end{thebibliography}


\end{document}